\documentclass[a4paper, 11pt, reqno]{amsart}
\usepackage{amsmath,amsfonts,amssymb,amsthm,enumerate}
\usepackage{graphicx}
 \usepackage{xy} \xyoption{all}

\newtheorem{thm}{Theorem}[section]
\newtheorem{cor}[thm]{Corollary}
\newtheorem{prop}[thm]{Proposition}
\newtheorem{lem}[thm]{Lemma}

\theoremstyle{definition}
\newtheorem{defn}[thm]{Definition}
\newtheorem{nota}[thm]{Notation}
\newtheorem{exas}[thm]{Examples}
\newtheorem{rem}[thm]{Remark}

\def\sm{\mathop{\sum}\limits_{i=1}^{m}}

\DeclareMathOperator{\End}{End}

\let\phi\varphi
\def\a{\alpha}

\def\d{\delta}
\def\m{\mu}
\def\n{\nu}

\begin{document}
\title{Anick type automorphisms and new irreducible representations of Leavitt path algebras}
\maketitle
\begin{center}
Shigeru Kuroda\footnote{Department of Mathematical Sciences, 
Tokyo Metropolitan University, Hachioji, Tokyo 192-0397 Japan. E-mail address: \texttt{kuroda@tmu.ac.jp }} and 
Tran Giang Nam\footnote{Institute of Mathematics, VAST, 18 Hoang Quoc Viet, Cau Giay, Hanoi, Vietnam. E-mail address: \texttt{tgnam@math.ac.vn}
			
\ \ {\bf Acknowledgements}: 
The first author is partly supported by JSPS KAKENHI Grant Number 18K03219.
The second author is supported by the Vietnam National Foundation for Science and Technology Development (NAFOSTED) under Grant 101.04-2020.01.} 
\end{center}

\begin{abstract} In this article, we give a new class of 
automorphisms of Leavitt path algebras of arbitrary graphs. Consequently, we obtain Anick type automorphisms of these Leavitt path algebras and new  
irreducible representations of Leavitt algebras of type $(1, n)$.
\medskip
	
	
\end{abstract}
\medskip
\section{Introduction}

Given a row-finite directed graph $E$ and any field $K$, Abrams and Aranda Pino in \cite{ap:tlpaoag05}, and independently Ara, Moreno, and Pardo in \cite{amp:nktfga}, introduced the \emph{Leavitt path algebra} $L_K(E)$.  Abrams and Aranda Pino later extended the definition in \cite{ap:tlpaoag08} to all countable directed graphs. Goodearl in \cite{g:lpaadl} extended the
notion of Leavitt path algebras $L_K(E)$ to all (possibly uncountable) directed graphs $E$. Leavitt path algebras generalize the Leavitt algebras $L_K(1, n)$ of \cite{leav:tmtoar}, and also contain many other interesting classes of algebras. In addition, Leavitt path algebras are intimately related to graph $C^*$-algebras (see \cite{r:ga}). During  the past fifteen years, Leavitt path algebras have become a topic of intense investigation by mathematicians from across the mathematical spectrum.  We refer the reader to \cite{a:lpatfd} and \cite{AAS:LPA} for a detailed history and overview of Leavitt path algebras.

The study of the module theory over Leavitt path algebras was initiated in \cite{AB:mtolpaakt}, in connection with some questions in algebraic $K$-theory. As an important step in the study of modules over a Leavitt path algebra $L_K(E)$, the study of the simple $L_K(E)$-modules has been investigated in numerous articles, see e.g. \cite{c:irolpa, Rang:ttopiolpaoag, ar:fpsmolpa, ar:lpawammir, Rang:osmolpa, amt:eosmolpa, Rang:lpawfpir, ahls:gsaatr,  ko:fdrolpa, anhnam}. 

Although, in general, classification of simple $L_K(E)$-modules seems to be a quite difficult task, recently there have been obtained a number of interesting results describing special classes of simple modules for Leavitt path algebras among which we mention, for example, the following ones. Following the ideas of Smith \cite{Smith2012}, Chen \cite{c:irolpa} constructed two types of simple modules $N_w$ and $V_{[p]}$ for the Leavitt path algebra $L_K(E)$ of an arbitrary graph $E$ by using various sinks $w$ in $E$ and the equivalence class $[p]$ of infinite paths tail-equivalent to a fixed infinite path $p$ in $E$, respectively. Chen's construction was extended by Ara and Rangaswamy \cite{ar:fpsmolpa} to introduce additive classes of non-isomorphic simple $L_K(E)$-modules $N_v^{B_{H(v)}}$, $N_v^{H(v)}$ and $V^f_{[c^{\infty}]}$ which associated respectively to both infinite emitters $v$ and pairs $(c, f)$ consisting of exclusive cycles $c$ together with irreducible polynomials $f \in K[x]\setminus\{1-x\}$. Rangaswamy \cite{Rang:osmolpa} constructed an additive class of simple $L_K(E)$-modules, $N_{v\infty}$, by using infinite emitters $v$. By a different method from those presented in \cite{ar:fpsmolpa}, \'{Anh} and the second author \cite{anhnam} constructed simple $L_K(E)$-modules $S^f_c$ associated to pairs $(c, f)$ consisting of simple closed paths $c$ together with irreducible polynomials $f$ in $K[x]$. We also should mention that  Ara and Rangaswamy \cite{ar:fpsmolpa} showed that all simple modules over the Leavitt path algebra of a finite graph in which every vertex is in at most one cycle are exactly $N_w$, $V_{[p]}$ and $V^f_{[c^{\infty}]}\cong S^f_c$, which are cited above. Ko\c{c} and \"{O}zaydin \cite{ko:fdrolpa} have classified all finite-dimensional modules for the Leavitt path algebra $L_K(E)$ of a row-finite graph $E$ via an explicit Morita equivalence given by an effective combinatorial (reduction) algorithm on the graph. Those obtained results induce our investigation to the study of simple modules for Leavitt path algebras of graphs having a vertex is at least in two cycles. The most important case of this class is the Leavitt path algebra of a rose with $n\ge 2$ petals---It is exactly the Leavitt algebra $L_K(1, n)$ (see, e.g., Proposition \ref{LA} below).

Following notes cited above, as of the writing of this article, there are two classes of non-isomorphic simple modules for Leavitt path algebras $L_K(R_n)$ of the rose $R_n$ with $n\ge 2$ petals:
\begin{itemize}
\item simple modules $V_{[p]}$ associated to infinite irrational paths $p$;
\item simple modules $S^f_c$ associated to pairs $(c, f)$ consisting of simple closed paths $c$ together with irreducible polynomials $f$ in $K[x]$.
\end{itemize}

Although Leavitt path algebras are studied by many researchers, little is known about their automorphisms. In this article, we give a new class of automorphisms of Leavitt path algebras, which includes analogues of the Anick automorphism of the free associative algebra $K\langle x_1,x_2,x_3\rangle $ (cf.~\cite[p.\ 343]{c:fratr}). The Anick automorphism is well known in the context of the Tame Generators Problems which asks if the automorphism group of the $K$-algebra $K\langle x_1,\ldots ,x_n\rangle $ is generated by the so-called elementary automorphisms. It is notable that, in 2007, Umirbaev~\cite{Umirbaev} solved this problem in the negative when $n=3$ and $K$ is of characteristic zero, by showing that the Anick automorphism cannot be obtained by composing elementary automorphisms. The main goal of this article is to construct additive classes of simple $L_K(R_n)$-modules, by studying the twisted modules of the simple modules $S^f_c$ under Anick type automorphisms of the Leavitt path algebras $L_K(R_n)$. 

The article is organized as follows. In Section 2, we provide a method to construct automorphisms of Leavitt path algebras of graphs (Theorem \ref{Anicktype} and Corollary~\ref{Anicktype1}). Consequently, we obtain Anick type automorphisms of these Leavitt path algebras (Corollaries~\ref{Anicktype2} and \ref{Anicktype3}).
In Section 3, based on Corollary~\ref{Anicktype3} and the simple modules $S^f_c$ mentioned above, we construct new classes of simple $L_K(R_n)$-modules (Theorems~\ref{Irrrep1} and \ref{Irrrep2}).

\section{Anick type automorphisms of Leavitt path algebras}
The aim of this section is to describe automorphisms of Leavitt path algebras of arbitrary graphs  (Theorem~\ref{Anicktype}). Consequently, we provide a method to construct automorphisms of unital Leavitt path algebras  in terms of invertible matrices (Corollary~\ref{Anicktype1}) and Anick type automorphisms of these Leavitt path algebras (Corollaries~\ref{Anicktype2} and \ref{Anicktype3}).

We begin this section by recalling some useful notions of graph theory. 
A (\textit{directed}) \textit{graph} is a quadruplet $E = (E^0, E^1, s, r)$ consists of two disjoint sets $E^0$ and $E^1$, called \emph{vertices} and \emph{edges}
respectively, together with two maps $s, r: E^1 \longrightarrow E^0$.  The vertices $s(e)$ and $r(e)$ are referred to as the \emph{source} and the \emph{range} of the edge~$e$, respectively. 
A vertex~$v$ for which $s^{-1}(v)$ is empty is called a \emph{sink}; a vertex~$v$ is \emph{regular} if $0< |s^{-1}(v)| < \infty$;  a vertex~$v$ is an \textit{infinite emitter} if $|s^{-1}(v)| = \infty$; and a vertex is \textit{singular} if it is either a sink or an infinite emitter. 

A \emph{finite path of length} $n$ in a graph $E$ is a sequence $p = e_{1} \cdots e_{n}$  of edges $e_{1}, \dots, e_{n}$ such that $r(e_{i}) = s(e_{i+1})$ for $i = 1, \dots, n-1$.  In this case, we say that the path~$p$ starts at the vertex $s(p) := s(e_{1})$ and ends at the vertex $r(p) := r(e_{n})$, we write $|p| = n$ for the length of $p$.  We consider the elements of $E^0$ to be paths of length $0$. We denote by $E^*$ the set of all finite paths in $E$.  An edge $f$ is an \textit{exit} for a path $p= e_1 \cdots e_n$ if $s(f) = s(e_i)$ but $f \neq e_i$ for some $1\le i\le n$. A finite path $p$ of positive length is called a \textit{closed path based at} $v$ if $v = s(p) = r(p)$. A \textit{cycle} is a closed path $p = e_{1} \cdots e_{n}$, and for which the vertices $s(e_1), s(e_2), \hdots, s(e_n)$ are distinct. A closed path $c$ in $E$ is called \textit{simple} if $c \neq d^n$ for any closed path $d$ and integer $n\ge 2$. We denoted by $SCP(E)$ the set of all simple closed paths in $E$.


\begin{defn}\label{LPA}
For an arbitrary graph $E = (E^0,E^1,s,r)$
and any  field $K$, the \emph{Leavitt path algebra} $L_{K}(E)$ {\it of the graph}~$E$ \emph{with coefficients in}~$K$ is the $K$-algebra generated by the union of the set $E^0$ and two disjoint copies of $E^1$, say $E^1$ and $\{e^*\mid e\in E^1\}$, satisfying the following relations for all $v, w\in E^0$ and $e, f\in E^1$:
\begin{itemize}
\item[(1)] $v w = \delta_{v, w} w$;
\item[(2)] $s(e) e = e = e r(e)$ and $e^*s(e)  = e^* = r(e) e^*$;
\item[(3)] $e^* f = \delta_{e, f} r(e)$;
\item[(4)] $v= \sum_{e\in s^{-1}(v)}ee^*$ for any regular vertex $v$;
\end{itemize}
where $\delta$ is the Kronecker delta.
\end{defn}
If $E^0$ is finite, then $L_K(E)$ is a unital ring having identity $1 = \sum_{v\in E^0}v$ (see, e.g. \cite[Lemma 1.6]{ap:tlpaoag05}).
It is easy to see that the mapping given by $v\longmapsto v$ for all $v\in E^0$, and $e\longmapsto e^*$, $e^*\longmapsto e$ for all $e\in E^1$, produces an involution on the algebra $L_K(E)$, and for any path $p = e_1e_2\cdots e_n$, the element $e^*_n\cdots e^*_2e^*_1$ of $L_K(E)$ is denoted by $p^*$.
It can be shown (\cite[Lemma 1.7]{ap:tlpaoag05}) that $L_K(E)$ is  spanned as a $K$-vector space by $\{pq^* \mid p, q\in F(E), r(p) = r(q)\}$. Indeed, $L_K(E)$ is a $\mathbb{Z}$-graded $K$-algebra: $L_K(E) = \oplus_{n\in \mathbb{Z}}L_K(E)_n$, where for each $n\in \mathbb{Z}$, the degree $n$ component $L_K(E)_n$ is the set $ \text{span}_K \{pq^*\mid p, q\in E^*, r(p) = r(q), |p|- |q| = n\}$.
Also, $L_K(E)$ has the following property: if $\mathcal{A}$ is a $K$-algebra generated by a family of elements $\{a_v, b_e, c_{e^*}\mid v\in E^0, e\in E^1\}$ satisfying the relations analogous to (1) - (4)  in Definition~\ref{LPA}, then there always exists a $K$-algebra homomorphism $\varphi: L_K(E)\longrightarrow \mathcal{A}$ given by $\varphi(v) = a_v$, $\varphi(e) = b_e$ and $\varphi(e^*) = c_{e^*}$.  We will refer to this property as the Universal Property of $L_K(E)$.

The following theorem provides us with a method to construct automorphisms of Leavitt path algebras of arbitrary graphs.

\begin{thm}\label{Anicktype}
Let $K$ be a field, $n$ a positive integer, $E$ a graph, and $v$ and $w$ vertices in $E$ (they may be the same). Let $e_1, e_2, \ldots, e_n$ be distinct edges in $E$ with $s(e_i) = v$ and $r(e_i) = w$ for all $1\le i\le n$. Let $P=(p_{i,j})$ and $Q = (q_{i, j})$ be elements of $M_n(L_K(E))$ such that $wP=Pw$, $wQ=Qw$ and $wPQ = wQP = wI_n$. Then the following statements hold:
	
{\rm (i)} There exists a unique homomorphism $\phi_{P, Q}:L_K(E)\to L_K(E)$ of $K$-algebras satisfying 
\begin{center}
$\phi_{P, Q}(u)=u,\quad \phi_{P, Q}(e)=e\quad\text{ and }\quad \phi_{P, Q}(e^*)=e^*$ \end{center} for all $u\in E^0$ and $e\in E^1\setminus \{ e_1,\ldots ,e_n\} $, and 
\begin{center}
$\phi_{P, Q}(e_i) = \sum^n_{k=1}e_kp_{k,i}$\quad	and \quad $\phi_{P, Q}(e^*_i) = \sum^n_{k=1}q_{i,k}e^*_k$
\end{center} for all $1\le i\le n$.

{\rm (ii)} If $w\phi_{P, Q}(p_{i,j})= wp_{i,j}$ for all $1\le i, j\le n$, or $w\phi_{P, Q}(q_{i,j})= wq_{i,j}$ for all $1\le i, j\le n$, then $\phi_{P, Q}$ is an isomorphism and $\phi_{P, Q}^{-1}=\phi_{Q, P}$. 
\end{thm}
\begin{proof} We first note that $wp_{k,i}=p_{k,i}w$ and $wq_{i,k}=q_{i,k}w$ for all $k, i$ (since $wP = Pw$ and $wQ=Qw$), and $\sum^n_{k=1}wp_{i,k}q_{k, j} = \delta_{i, j}w = \sum^n_{k=1}wq_{i,k}p_{k, j}$ for all $i, j$, where $\delta$ is the Kronecker delta.
	
	(i) We define the  elements $\{Q_u \ | \ u\in E^0\}$ and $\{T_e, T_{e^*} \ | \ e\in E^1\}$ of $L_K(E)$  by setting $Q_u = u$, 	
	\begin{equation*}
	T_e=  \left\{
	\begin{array}{lcl}
	\sum^n_{k=1}e_kp_{k,i}&  & \text{if } e= e_i \text{ for some } 1\le i\le n\\
	e&  & \text{otherwise},  
	\end{array}%
	\right.
	\end{equation*}%
	\medskip
	and 
	\begin{equation*}
	T_{e^*}=  \left\{
	\begin{array}{lcl}
	\sum^n_{k=1}q_{i,k}e^*_k&  & \text{if } e= e_i \text{ for some } 1\le i\le n\\
	e^*&  & \text{otherwise}.
	\end{array}%
	\right.
	\end{equation*}%
	\medskip
	
	\noindent
	We claim that $\{Q_u, T_e, T_{e^*}\mid u\in E^0, e\in E^1\}$ is a family in $L_K(E)$ satisfying the relations analogous to (1) - (4) in Definition~\ref{LPA}. Indeed, we have $Q_u Q_{u'} = u u' = \delta_{u, u'}u = \delta_{u, u'} Q_u$ for all $u, u'\in E^0$, showing relation (1).
	
	For (2), we always have $Q_{s(e)}T_e = T_e = T_eT_{r(e)}$ and $T_{e^*} Q_{s(e)} = T_{e^*} = Q_{r(e)}T_{e^*}$ for all $e\in E^1\setminus \{ e_1,\ldots ,e_n\}$. For each $1\le i\le n$, since 
	$$ve_k=e_kw=e_k,\ \ we_k^*=e_k^*v=e_k^*,\ \ wp_{k,i}=p_{k,i}w\ \ \text{and}\ \ wq_{i,k}=q_{i,k}w$$ for all $k$, we have 
	\begin{align*}
	Q_vQ_{e_i}&=v\sum _{k=1}^ne_kp_{k,i}=\sum _{k=1}^ne_kp_{k,i}=Q_{e_i}\\
	Q_{e_i}Q_w&=\sum _{k=1}^ne_kp_{k,i}w
	=\sum _{k=1}^ne_kwp_{k,i}
	=\sum _{k=1}^ne_kp_{k,i}=Q_{e_i}\\
	Q_wT_{e^*_i}&=w\sum _{k=1}^nq_{i,k}e_k^*
	=\sum _{k=1}^nq_{i,k}we_k^*
	=\sum _{k=1}^nq_{i,k}e_k^*=T_{e^*_i}\\
	T_{e^*_i}Q_v&=\sum _{k=1}^nq_{i,k}e_k^*v
	=\sum _{k=1}^nq_{i,k}e_k^*=T_{e^*_i}. 
	\end{align*}
	
	For (3), we obtain that $T_{e^*}T_f = e^* f = \delta_{e, f}r(e)$ for all $e, f\in E^1\setminus \{ e_1,\ldots ,e_n\}$. For each $f\in E^1\setminus \{ e_1,\ldots ,e_n\} $ and $1\le i\le n$, we have 
	\begin{center}
		$T_{e^*_i}T_{f}=\sum _{k=1}^nq_{i,k}e_k^*f=0$\quad and\quad $
		T_{f^*}T_{e}=\sum _{k=1}^nf^*e_kp_{k,i}=0,$	
	\end{center}
	since $e_k^*f=f^*e_k=0$. For $i,j\in \{ 1,\ldots ,n\}$, we have 
	\begin{align*}
	&T_{e^*_i}T_{e_j}=\sum _{k=1}^n\sum _{l=1}^nq_{i,k}e_k^*e_lp_{l,j}
	=\sum _{k=1}^n\sum _{l=1}^nq_{i,k}\delta _{k,l}wp_{l,j} 
	=\sum _{k=1}^nwq_{i,k}p_{k,j}=\delta_{i,j}w =\delta_{i,j}Q_w, 
	\end{align*}
	since $e_k^*e_l=\delta _{k,l}w$ and $wp_{l,j}=p_{l,j}w$. 
	
	For (4), let $u$ be a regular vertex in $E$. If $u \neq v$, then $\sum_{e\in s^{-1}(u)}T_eT_{e^*} = \sum_{e\in s^{-1}(u)}ee^* = u = Q_u$. Consider the case when $u=v$, that is, $v$ is a regular vertex. Write $$s^{-1}(v)=\{ e_1,\ldots ,e_n,e_{n+1},\ldots, e_m\}$$
	for some distinct $e_{n+1},\ldots ,e_m\in E^1$ with $n\le m<\infty $. We note that 
	\begin{align*}
	\sum _{i=1}^nT_{e_i}T_{e^*_i}&=\sum _{i=1}^n\sum _{k=1}^n\sum _{l=1}^n e_kp_{k,i} q_{i,l}e^*_l=\sum _{k=1}^n\sum _{l=1}^n e_kw(\sum^n_{i=1}p_{k,i} q_{i,l})e^*_l=\\
	& =\sum _{k=1}^n\sum _{l=1}^n e_k(\delta_{k,l} w)e^*_l=\sum_{k=1}^ne_ke_k^*,
	\end{align*}
	and so, we have 
	$$\sum _{e\in s^{-1}(v)}T_e T_{e^*}	=\sum _{i=1}^mT_{e_i} T_{e^*_i}=
	\sum _{i=1}^me_ie_i^*=v=Q_v,$$ thus showing the claim. Then, by the Universal Property of $L_K(E)$, there exists a $K$-algebra homomorphism $\phi_{P, Q}: L_K(E)\longrightarrow L_K(E)$, which maps $u\longmapsto Q_u$, $e\longmapsto T_e$ and $e^*\longmapsto T_{e^*}$,  as desired.
	
(ii) Let $P'$ and $Q'$ be elements of $M_n(L_K(E))$ obtained from $P$ and $Q$ by applying the homomorphism $\phi_{P, Q}$, respectively. Assume that  $w\phi_{P, Q}(p_{i,j})=wp_{i,j}$ for all $1\le i, j\le n$. 	Then, since $\phi_{P, Q}$ is a $K$-algebra homomorphism and $wPQ =  wI_n$, we have $wP'Q' = wI_n$ and $wPQ' = wI_n$. This implies that \[wQ' = wQPQ' = QwPQ'= QwI_n = wQI_n = wQ,\] that means, $wq_{i, j} = w\phi_{P, Q}(q_{i, j})$ for all $i, j$. Similarly, we receive the fact that if $wQ' = wQ$ then $wP' = wP$. Therefore, in any case, we have both $wP' = wP$ and $wQ' = wQ$.

We claim that $\phi_{P, Q}\phi_{Q, P} = id_{L_K(E)}$. Indeed, it suffices to check that 
\begin{center}
$\phi_{P, Q}\phi_{Q, P}(e_i) = e_i$ and $\phi_{P, Q}\phi_{Q, P}(e^*_i) = e^*_i$ for all $1\le i\le n$.	
\end{center}
 For each $1\le i\le n$, by definition of $\phi_{Q, P}$, $\phi_{Q, P}(e_i) = \sum^n_{k=1}e_kq_{k, i}=\sum^n_{k=1}e_kwq_{k, i}$ and $\phi_{Q, P}(e^*_i) = \sum^n_{k=1}p_{i, k}e^*_k=\sum^n_{k=1}p_{i, k}we^*_k=\sum^n_{k=1}wp_{i, k}e^*_k$,  so 
\begin{align*}
\phi_{P, Q}\phi_{Q, P}(e_i) &=\phi_{P, Q}(\sum^n_{k=1}e_kwq_{k, i})=\sum^n_{k=1}\phi_{P, Q}(e_k)w\phi_{P, Q}(q_{k,i})=\sum^n_{k=1}\sum^n_{l=1}e_lp_{l,k}wq_{k,i}\\
&=\sum^n_{l=1}e_l(\sum^n_{k=1}wp_{l,k}q_{k,i})= \sum^n_{l=1}e_l\delta_{l, i}w= e_iw= e_i
\end{align*} 
and 
\begin{align*}
\phi_{P, Q}\phi_{Q, P}(e^*_i)&=\phi_{P, Q}(\sum^n_{k=1}wp_{i, k}e^*_k)=\sum^n_{k=1}w\phi_{P, Q}(p_{i, k})\phi_{P, Q}(e^*_k)= \sum^n_{k=1}\sum^n_{l=1}wp_{i,k}q_{k,l}e^*_l\\
& =\sum^n_{l=1}(\sum^n_{k=1}wp_{i,k}q_{k,l})e^*_l= \sum^n_{l=1}\delta_{i, l}we^*_l= we^*_i = e^*_i,
\end{align*} 
proving the claim. This implies that $\phi_{P, Q}$ is surjective.
	
We next prove that $\phi_{P, Q}$ is injective. To the contrary, suppose there exists a nonzero element $x\in\ker(\phi_{P, Q})$. Then, by the Reduction Theorem (see, e.g., \cite[Theorem 2.2.11]{AAS:LPA}), there exist $a, b \in L_K(E)$ such that either $axb = u \neq 0$ for some $u\in E^0$, or $axb = p(c)\neq 0$, where $c$ is a cycle in $E$ without exits and $p(x)$ is a nonzero polynomial in $K[x, x^{-1}]$.
	
In the first case, since $axb\in \ker(\phi_{P, Q})$, this would imply that $u= \phi_{P, Q}(u) = 0$ in $L_K(E)$; but each vertex is well-known to be a nonzero element inside the Leavitt path algebra, which is a contradiction.
	
So we are in the second case: there exists a cycle $c$ in $E$ without exits such that $axb = \sum^m_{i=-l}k_ic^i\neq 0$, where $k_i\in K$, $l$ and $m$ are nonnegative integers, and we interpret $c^i$ as $(c^*)^{-i}$ for negative $i$, and we interpret $c^0$ as $u := s(c)$. Write $c = g_1q_2\cdots g_t$, where $g_i\in E^1$ and $t$ is a positive integer. If $g_i\in E^1\setminus \{ e_1,\ldots ,e_n\}$ for all $1\le i\le t$, then $\phi_{P, Q}(c) = c$ and $\phi_{P, Q}(c^*) = c^*$, so $0\neq \sum^m_{i=-l}k_ic^i = \sum^m_{i=-l}k_i\phi_{P, Q}(c^i) = \phi_{P, Q}(axb) = 0$ in $L_K(E)$, a contradiction. Consider the case that there exists a $1\le k\le t$ such that $g_k = e_i$ for some $i$. Then, since $c$ is a cycle without exits, we must have $n = 1$ and $k$ is a unique element such that $g_k = e_1$. Let $\alpha := g_{k+1}\cdots g_tg_1\cdots g_{k-1}e_1$. We have that $\alpha$ is a cycle in $E$ without exits and $s(\alpha) = w$. Since $n=1$, $P= p_{1, 1}$ and $Q= q_{1, 1}$ are two elements of $L_K(E)$ with $wp_{1, 1}q_{1, 1} = w= wq_{1, 1}p_{1, 1}$, so $wp_{1,1}w$ is also a unit of $wL_K(E)w$ with $(wp_{1,1}w)^{-1} = w q_{1,1}w$. Moreover, $\phi_{P, Q}(wp_{1,1}w) = \phi_{P, Q}(w)\phi_{P, Q}(p_{1,1})\phi_{P, Q}(w) =w\phi_{P, Q}(p_{1,1})w= wp_{1,1}w$ and $\phi_{P, Q}(wq_{1,1}w) = \phi_{P, Q}(w)\phi_{P, Q}(q_{1,1})\phi_{P, Q}(w) = w\phi_{P, Q}(q_{1,1})w=wq_{1,1}w$. By \cite[Lemma 2.2.7]{AAS:LPA}, we have \[wL_K(E)w = \{\sum^h_{i=l}k_i\alpha^i\mid k_i\in K, l\le h, h, l\in \mathbb{Z}\}\cong K[x, x^{-1}]\]via an isomorphism that sends $v$ to $1$, $\alpha$ to $x$ and $\alpha^*$ to $x^{-1}$, and so $wp_{1,1}w = a\alpha^s$ and $wq_{1,1}w = a^{-1}\alpha^{-s}$ for some $a\in K\setminus \{0\}$ and $s\in \mathbb{Z}$. If $s\ge 0$, then 
\begin{align*}
a\alpha^s&=wp_{1,1}w= \phi_{P, Q}(wp_{1,1}w) = \phi_{P, Q}(a \alpha^s) = a \phi_{P, Q}(\alpha)^s=\\
&= a (\phi_{P, Q}(g_{k+1}\cdots g_tg_1\cdots g_{k-1}e_1))^s=a (g_{k+1}\cdots g_tg_1\cdots g_{k-1}e_1p_{1,1})^s =\\
&= a (g_{k+1}\cdots g_tg_1\cdots g_{k-1}e_1 wp_{1,1}w)^s = a^{s+1}\alpha^{s(s+1)}
\end{align*}
in $wL_K(E)w$, so $s = 0$, that is, $wp_{1,1}w = aw$ and $wq_{1,1}w = a^{-1}w$. If $s\le 0$, then since $\phi_{P, Q}(wq_{1,1}w) = wq_{1,1}w$, and by repeating the argument described
in the first case,  we obtain that $s = 0$, $wp_{1,1}w = aw$ and $wq_{1,1}w = a^{-1}w$. This implies that 
\begin{align*}
\phi_{P, Q}(c) & =\phi_{P, Q}(g_1\cdots g_{k-1}e_1g_{k+1}\cdots g_t) =(g_1\cdots g_{k-1})e_1p_{1,1} (g_{k+1}\cdots g_t)=\\
&=(g_1\cdots g_{k-1}e_1)wp_{1,1}w (g_{k+1}\cdots g_t)= (g_1\cdots g_{k-1}e_1)aw(g_{k+1}\cdots g_t) = ac,
\end{align*} 
and 
\begin{align*}
\phi_{P, Q}(c^*) & =\phi_{P, Q}(g^*_{t}\cdots g^*_{k+1}e^*_1g^*_{k-1}\cdots g^*_{1}) =(g^*_{t}\cdots g^*_{k+1})q_{1,1}e^*_1 (g^*_{k-1}\cdots g^*_{1})=\\
&=(g^*_{t}\cdots g^*_{k+1})wq_{1,1}w(e^*_1 g^*_{k-1}\cdots g^*_{1})= (g^*_{t}\cdots g^*_{k+1})a^{-1}w(e^*_1 g^*_{k-1}\cdots g^*_{1})=\\& = a^{-1}c^*,
\end{align*} 
so $\phi_{P, Q}(c^l) = a^lc^l$ for all $l\in \mathbb{Z}$. We then have $0\neq \sum^m_{i=-l}k_ia^ic^i = \sum^m_{i=-l}k_i\phi_{P, Q}(c^i) = \phi_{P, Q}(axb) = 0$ in $L_K(E)$, which is a contradiction.
	
In any case, we arrive at a contradiction, and so we infer that $\phi_{P, Q}$ is injective, thus $\phi_{P, Q}$ is an isomorphism with $\phi_{P, Q}^{-1}=\phi_{Q, P}$, finishing the proof.
\end{proof}

Consequently, we obtain a method to construct automorphisms of unital Leavitt path algebras  in terms of invertible matrices.

\begin{cor}\label{Anicktype1}
Let $K$ be a field, $n$ a positive integer, $E$ a graph with finitely many vertices, and $v$ and $w$ vertices in $E$ (they may be the same). Let $e_1, e_2, \ldots, e_n$ be distinct edges in $E$ with $s(e_i) = v$ and $r(e_i) = w$ for all $1\le i\le n$. Let $P=(p_{i,j})$ be a unit of $M_n(L_K(E))$ with $wP=Pw$ and $P^{-1} = (q_{i, j})$. Then the following statements hold:
	
{\rm (i)} There exists a unique homomorphism $\phi_P:L_K(E)\to L_K(E)$ of $K$-algebras satisfying
\begin{center}
$\phi _P(u)=u,\quad \phi _P(e)=e\quad\text{ and }\quad \phi _P(e^*)=e^*$ \end{center} for all $u\in E^0$ and $e\in E^1\setminus \{ e_1,\ldots ,e_n\} $, and 
\begin{center}
$\phi_P(e_i) = \sum^n_{k=1}e_kp_{k,i}$\quad	and \quad $\phi_P(e^*_i) = \sum^n_{k=1}q_{i,k}e^*_k$
\end{center} for all $1\le i\le n$.
	
{\rm (ii)} If $\phi _P(p_{i,j})=p_{i,j}$ for all $1\le i, j\le n$, 
then $\phi _P$ is an isomorphism and $\phi _P^{-1}=\phi _{P^{-1}}$. 
\end{cor}
\begin{proof}  Since $wP = Pw$, we have $P^{-1}wPP^{-1} = P^{-1}PwP^{-1}$, so $wP^{-1}=P^{-1}w$. Since $PP^{-1}= I_n = P^{-1}P$, it is obvious that $wPP^{-1}= wI_n = wP^{-1}P$. Therefore, the pair of the matrices $P$ and $P^{-1}$ satisfies the conditions analogous to the one of the matrices $P$ and $Q$ in Theorem \ref{Anicktype}. Then, by Theorem \ref{Anicktype}, we immediately obtain the statements, thus finishing the proof.
\end{proof}

For clarification, we illustrate Theorem \ref{Anicktype} and Corollary~\ref{Anicktype1} by presenting the following example.

\begin{exas}
Let $K$ be a field and $R_1$ the following graph $$R_1 = 
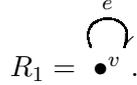
\xymatrix{\bullet^{v}\ar@(ul,ur)^e}.$$ Then $L_K(R_1)\cong K[x, x^{-1}]$ via an isomorphism that sends $v$ to $1$, $e$ to $x$ and $e^*$ to $x^{-1}$.	Let $P= e^*$. We have that $P$ is a unit of $L_K(R_1)$ with $P^{-1} = e$. Then, by Corollary~\ref{Anicktype1}, we obtain the endomorphism $\phi_P$ defined by: $v\longmapsto v$, $e\longmapsto eP = ee^* = v$ and $e^*\longmapsto P^{-1}e^* = ee^* = v$. We have that $\phi_P$ is not isomorphic and
$\phi_P(P) = \phi_P(e^*) = v \neq e^*=P$ in $L_K(R_1)$. This implies the hypothesis ``$\phi _P(p_{i,j})=p_{i,j}$ for all $1\le i, j\le n$" in part (ii) of Corollary~\ref{Anicktype1}
cannot be removed.
\end{exas}

In light of the well-known Anick automorphism (see \cite[p.\ 343]{c:fratr}) of the free associative algebra $K\langle x, y, z\rangle$, we construct Anick type automorphisms of unital Leavitt path algebras.

\begin{cor}[Anick type automorphism]\label{Anicktype2}
Let $K$ be a field, $E$ a graph with finitely many vertices, and $v$ and $w$ vertices in $E$ (they may be the same). Let $e_1$ and $e_2$ be two distinct edges in $E$ with $s(e_i) = v$ and $r(e_i) = w$ for all $i$. Let $A_E(e_1, e_2)$ be the $K$-subalgebra of $L_K(E)$ generated by the sets $E^0$, $E^1\setminus \{e_2\}$ and $\{e^*\mid e\in E^1\setminus \{e_1\}\}$. Then, for any $p\in A_E(e_1, e_2)$ with $wp = pw$, 
there exists a unique automorphism $\sigma_p$ 
of the $K$-algebra $L_K(E)$ satisfying 
$$
\sigma _p(e_2)=e_2+e_1p,\ \sigma_p(e_1^*)=e_1^*-pe_2^* ,\ 
\sigma^{-1}_p(e_2)=e_2 - e_1p,\ \sigma^{-1}_p(e_1^*)=e_1^* +pe_2^*
$$
and $\sigma_p(q) = q$ for all $q\in A_E(e_1, e_2)$. 
\end{cor}
\begin{proof}
Let $P=\begin{pmatrix}
1& p \\
0& 1
\end{pmatrix}\in M_2(L_K(E))$. We then have that $P$ is a unit of $M_2(L_K(E))$ with $P^{-1}=\begin{pmatrix}
1& -p \\
0& 1
\end{pmatrix}.$ It is clear that $\sigma_p = \phi_P$, which is described in Corollary~\ref{Anicktype1} (i), and $\phi_P(q) = q$ for all $q\in A_E(e_1, e_2)$. Then, using Corollary~\ref{Anicktype1}, we immediately receive the corollary, thus finishing the proof.
\end{proof}

Let $K$ be a field and $n \ge 2$ any integer. Then the  \textit{Leavitt $K$-algebra of type} $(1;n)$, denoted by $L_K(1, n)$,  is the $K$-algebra \[K\langle x_1, \hdots, x_n, y_1, \hdots, y_n\rangle/\langle \sum^{n}_{i=1}x_iy_i -1, y_ix_j - \delta_{i,j}1\mid 1\le i, j\le n \rangle.\] Notationally, it is often more convenient to view $L_K(1, n)$  as the free associative $K$-algebra on the $2n$ variables $x_1, \hdots, x_n, y_1, \hdots, y_n$ subject to the relations $ \sum^{n}_{i=1}x_iy_i =1$ and $y_ix_j = \delta_{i,j}1 \, (1\le i, j\le n)$; see \cite{leav:tmtoar} for more details.

For any integer $n\ge 2$, we let $R_n$ denote the \textit{rose with $n$ petals} graph having one vertex and $n$ loops:
$$R_n = 
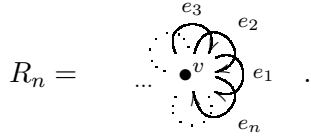
\xymatrix{ & {\bullet^v} \ar@(ur,dr)^{e_1}  \ar@(u,r)^{e_2} \ar@(ul,ur)^{e_3}  \ar@{.} @(l,u) \ar@{.} @(dr,dl)
\ar@(r,d)^{e_n}  \ar@{}[l] ^{\hdots} } \ \ .$$ Then $L_K(R_n)$ 	is defined to be the $K$-algebra generated by $v$, $e_1, \ldots, e_n$, $e^*_1, \ldots, e^*_n$, satisfying the following relations 
\begin{center}
$v^2 =v, ve_i = e_i= e_iv$, $ve^*_i = e^*_i = e^*_i v$, $e^*_i e_j = \delta_{i,j} v$ and $\sum^n_{i=1}e_ie^*_i =v$	
\end{center}
for all $1\le i, j\le n$. In particular $v = 1_{L_K(R_n)}$.

\begin{prop}[{\cite[Proposition 1.3.2]{AAS:LPA}}]\label{LA}
Let $n\ge 2$ be any positive integer, $K$ a field and $R_n$ the rose with petals. Then $L_K(1, n)\cong L_K(R_n)$ as $K$-algebras.	
\end{prop}
\begin{proof}
We can show that the map $\phi: L_K(1, n)\longrightarrow L_K(R_n)$, given by the extension of $\phi(1) = v$, $\phi(x_i) = e_i$	and $\phi(y_i) = e^*_i$, is a $K$-algebra isomorphism. 
\end{proof}

With Proposition~\ref{LA} in mind, for the remainder of this article we investigate the structure of the Leavitt algebra $L_K(1, n)$ by equivalently investigating the structure of the Leavitt path algebra 
$L_K(R_n)$.

\begin{nota}
For any integer $n\ge 2$ and any field $K$, we denote by $A_{R_n}(e_1, e_2)$ the $K$-subalgebra of $L_K(R_n)$ generated by $$v, e_1, e_3, \hdots , e_n, e^*_2, \hdots , e^*_n.$$	
\end{nota}
We should mention that by \cite[Theorem 1]{aajz:lpaofgkd}, the following elements form a basis of the $K$-algebra $A_{R_n}(e_1, e_2)$: (1) $v$, (2) $p= e_{k_1}\cdots e_{k_m}$, where $k_i \in \{1, 3, \hdots , n\}$, (3) $q^* = e^*_{t_1}\cdots  e^*_{t_h}$, where $t_i \in \{2, 3, \hdots , n\}$,
(4) $pq^*$, where $p$ and $q^*$ are defined as in items (2) and (3), respectively.
\medskip

The following result provides us with Anick type automorphisms of Leavitt algebras of type $(1, n)$.

\begin{cor}\label{Anicktype3}
Let $n\ge 2$ be a positive integer, $K$ a field and $R_n$ the rose with $n$ petals. Then, for any $p \in A_{R_n}(e_1, e_2)$, 
there exists a unique automorphism $\sigma_p$ of the $K$-algebra $L_K(R_n)$ 
satisfying 
$\sigma_p(e_2) = e_2 + e_1p$, 
$\sigma_p(e^*_1) = e^*_1 - pe^*_2$, 
$\sigma^{-1}_p(e_2) = e_2 - e_1p$, 
$\sigma^{-1}_p(e^*_1) = e^*_1 + pe^*_2$ 
and $\sigma_p(q) = q$ for all $q\in A_{R_n}(e_1, e_2)$. 
\end{cor}
\begin{proof}
Since $vp =p =pv$, the corollary immediately follows from Corollary~\ref{Anicktype2}.	
\end{proof}

\section{New irreducible representations of $L_K(R_n)$}	

In this section, we study the twisted modules of the simple $L_K(R_n)$-modules $S^f_c$ mentioned in the Introduction under Anick type automorphisms of $L_K(R_n)$ introduced in  Corollary~\ref{Anicktype3}. In particular, we obtain  new classes of simple $L_K(R_n)$-modules (Theorems~\ref{Irrrep1} and \ref{Irrrep2}).

Let $E$ be an arbitrary graph. An \textit{infinite path} $p:= e_1\cdots e_n\cdots$ in a graph $E$ is a sequence of edges $e_1, \hdots, e_n, \hdots $ such that $r(e_i) = s(e_{i+1})$ for all $i$. We denote by $E^{\infty}$ the set of all infinite paths in $E$.
For $p:= e_1\cdots e_n\cdots\in E^{\infty}$  and $n\ge 1$, Chen ([2]) defines $\tau_{> n}(p) = e_{n+1}e_{n+2}\cdots,$ and $\tau_{\le n}(p) = e_1e_2\cdots e_n$. Two infinite paths $p, q$ are said to be \textit{tail-equivalent} (written $p\sim q$) if there exist positive integers $m, n$ such that $\tau_{> n}(p) = \tau_{> m}(q)$. Clearly $\sim$ is an equivalence relation on $E^{\infty}$, and we let $[p]$ denote the $\sim$ equivalence class of the infinite path $p$.

Let $c$ be a closed path in $E$. Then the path $c c c\cdots$ is an infinite path in $E$, which we denote by $c^{\infty}$. Note that if $c$ and $d$ are closed paths in $E$ such that $c = d^n$, then $c^{\infty}=d^{\infty}$ as elements of $E^{\infty}$. The infinite path $p$ is called \textit{rational} in case $p\sim c^{\infty}$ for some closed path $c$. If $p\in E^{\infty}$ is not rational we say $p$ is \textit{irrational}. We denote by $E^{\infty}_{rat}$ and $E^{\infty}_{irr}$ the sets of rational and irrational paths in $E$, respectively.

Given a field $K$ and an infinite path $p$, Chen ([2]) defines $V_{[p]}$ to be the $K$-vector space having $\{q\in E^{\infty}\mid q\in [p]\}$ as a basis, that is, having basis consisting of distinct elements of $E^{\infty}$ which are tail-equivalent to $p$. $V_{[p]}$ is made a left $L_K(E)$-module by defining, for all $q\in [p]$ and all $v\in E^0$, $e\in E^1$, 

$v\cdot q = q$ or $0$ according as $v = s(q)$ or not;

$e \cdot q = eq$ or $0$ according as $r(e) = s(q)$ or not;

$e^* \cdot q = \tau_1(q)$ or $0$ according as $q = e\tau_1(q)$ or not.\\ 
In \cite[Theorem 3.2]{c:irolpa} Chen showed the following result.

\begin{thm}[{\cite[Theorem 3.2]{c:irolpa}}]\label{Chenmod} Let $K$ be a field, $E$ an arbitrary graph and $p, \, q\in E^{\infty}$. Then the following holds:
	
$(1)$	$V_{[p]}$ is a simple left $L_K(E)$-module;
	
$(2)$ $End_K(V_{[p]})\cong K$;
	
$(3)$ $V_{[p]} \cong V_{[q]}$ if and only if $p \sim q$, which happens precisely when $V_{[p]} = V_{[q]}$.
\end{thm}

Theorem \ref{Chenmod} provides us with the following two classes of simple modules for the Leavitt path algebra $L_K(E)$ of an arbitrary graph $E$:

\begin{itemize} 	
\item $V_{[\alpha]}$, where $\alpha \in E^{\infty}_{irr}$;
\item $V_{[\beta]}$, where $\beta \in E^{\infty}_{rat}$.
\end{itemize}
We note that for any $\beta \in E^{\infty}_{rat}$,  $V_{[\beta]} = V_{[c^{\infty}]}$ for some $c \in SCP(E)$. By \cite[Theorem 2.8]{amt:eosmolpa}, we have
$V_{[\beta]} =V_{[c^{\infty}]} \cong L_K(E)v/L_K(E)(c -v)$ as left $L_K(E)$-modules, {i.e.,} it is finitely presented; while $V_{[\alpha]}$ $(\alpha \in E^{\infty}_{irr})$ is, in general, not finitely presented by \cite[Corollary 3.5]{anhnam}. 
In \cite{anhnam} \'{A}nh and the second author constructed simple $L_K(E)$-modules $S^f_c$ associated to pairs $(f, c)$ consisting of simple closed paths $c$ together with irreducible polynomials $f$ in $K[x]$. We will represent again this result in Theorem~\ref{Rangmod} below. To do so, we need some notions.

Let $K$ be a field and $E$ a graph and $c$ a closed path in $E$ based at $v$. Let $f(x) = a_0 + a_1 x + \cdots + a_n x^n$ be a polynomial in $K[x]$. We denote by $f(c)$ the element \[f(c):= a_0v + a_1c + \cdots + a_nc^n\in L_K(E).\] We denote by $K[c]$ the subalgebra of $L_K(E)$ generated by $v$ and $c$. By the $\mathbb{Z}$-grading on $L_K(E)$, $K[c]$ is isomorphic to the polynomial algebra $K[x]$ by the map: $v\longmapsto 1$ and $c\longmapsto x$. We denote by $\text{Irr}(K[x])$ the set of all irreducible polynomials in $K[x]$ written in the form $1 - a_1x - \cdots - a_n x^n$.

\begin{thm}[{cf. \cite[Theorems 4.3 and 4.7]{anhnam}}]\label{Rangmod} Let $K$ be a field, $E$ an arbitrary graph, $c$ a simple closed path in $E$ based at $v$, and $f(x) = 1 - a_1x - \cdots - a_n x^n$ an irreducible polynomial in $K[x]$. Then the following holds:
	
$(1)$ The cyclic left $L_K(E)$-module $S^f_c$ generated by $z$ subject to $z= (a_1c + \cdots + a_nc^n)z$, is simple, and its endomorphism ring is isomorphic to $K[x]/K[x]f(x)$. Moreover, \[S^f_c\cong L_K(E)v/L_K(E)f(c),\] as left $L_K(E)$modules, via the map $z\longmapsto v + L_K(E)f(c)$;

$(2)$ For any $g\in \rm{Irr}$$(K[x])$ and any simple closed path $d$ in $E$, $S^f_c\cong S^g_d$ as left $L_K(E)$-modules if and only if $f = g$ and $c^{\infty} \sim d^{\infty}$.
\end{thm}
\begin{proof}
(1) We note that $vz = z$ and $z = c^nf_1(c)^nz$ for all $n\ge 1$, where $f_1(c) = a_1v + a_2c + \cdots + a_n c^{n-1}$. By the $\mathbb{Z}$-grading on $L_K(E)$,  $L_K(E)v/L_K(E)f(c)\neq 0$. Since $f(c) (v + L_K(E)f(c)) = f(c) + L_K(E)f(c) =0$ in $L_K(E)v/L_K(E)f(c)$, there exists a surjective $L_K(E)$-homomorphism $\theta: S^f_c\longrightarrow L_K(E)v/L_K(E)f(c)$ such that $\theta(z) = v + L_K(E)f(c)$, and so $S^f_c\neq 0$. 

We claim that $S^f_c$ is a simple left $L_K(E)$-module. Indeed, let $y$ be a nonzero element in $S^f_c$. Since $S^f_c = L_K(E)z$, $y$ may be written in  the form $y = rz$ and $0\neq r=\sm k_i\m_i\n_i^*\in L_K(E)$, where $m$ is minimal such that $k_i \in K\setminus \{0\}$ and $\m_i, \n_i \in E^*$ with $r(\m_i) = r(\n_i)$ for all $1\le i\le m$. Let $n$ be a positive integer such that $|\n_i\le n|c|$ for all $1\le i\le m$. We then have

$$y=(\sm k_i\m_i\n^*_i)z=(\sm k_i\m_i\n^*_i)c^nf_1(c)^nz=(\sm k_i\m_i\n^*_ic^nf_1(c)^n)z.$$
By the minimality of $m$, $\n^*_ic^n\neq 0$ for all $1\le i\le m$. Then, for each $i$, there exists $\d_i\in E^*$ such that $c^n = \n_i\d_i$ and $r(\d_i) = v:= r(c)= s(c)$. This implies

$$y=(\sm k_i\m_i\n^*_ic^nf_1(c)^n)z=(\sm k_i\m_i\d_if_1(c)^n)z=\sm k_i\a_if_1(c)^nz$$
where $\a_i=\m_i\d_i \,(i=1, \cdots , m)$. We note that $\alpha_i c^{\infty} = \alpha_j c^{\infty}$ in $E^{\infty}$ if and only if $\alpha_i = \alpha_j c^{n_i}$ for some $n_i\in \mathbb{Z}^+$, or $\alpha_j = \alpha_i c^{n_j}$ for some $n_j\in \mathbb{Z}^+$, and $f(c) z= 0$ in $S^f_c$. Consequently, $y$ may be written in the form 
$y = \sum^d_{i=1}\beta_i p_i(c)z$, where $p_i(c)$'s are nonzero elements in $K[c]/K[c]f(c)$ and $\beta_i$'s are paths in $E^*$ such that $\beta_ic^{\infty}$'s are distinct infinite paths in $E^{\infty}$, and so there exists a positive integer $t$ such that $\tau_{\le t}(\beta_ic^{\infty})$'s are distinct paths in $E^*$. This implies that $\tau_{\le t + j}(\beta_ic^{\infty})$'s are distinct paths in $E^*$ for all $j \ge 0$. Therefore,  without loss of generality, we may assume that $\tau_{\le t}(\beta_1c^{\infty}) = \beta_1 c^{l}$ for some $l \ge 1$. We then have $\tau_{\le t}(\beta_1c^{\infty})^*y = \tau_{\le t}(\beta_1c^{\infty})^*(\sum^d_{i=1}\beta_i p_i(c)z) = \tau_{\le t}(\beta_1c^{\infty})^*(\sum^d_{i=1}\beta_ic^l p_i(c)f_1(c)^lz) = p_1(c)f_1(c)^lz$. Since $p_1(c)f_1(c)^l$ is a nonzero element in $K[c]/K[c]f(c)$ and $K[c]/K[c]f(c)\cong K[x]/K[x]f(x)$ is a field, there exist $q$ and $h\in K[c]$ such that $q p_1(c)f_1(c)^l = v + hf(c)$, and so $$q \tau_{\le t}(\beta_1c^{\infty})^*y = q p_1(c)f_1(c)^lz = v z + hf(c) z = z.$$
This implies $z \in L_K(E) y$, and hence $S^f_c = L_K(E)z = L_K(E)y$.
Consequently, $S^f_c$ is a simple left $L_K(E)$-module, showing the claim. This implies that $\theta$ is an isomorphism.

Let $\phi: S^f_c\longrightarrow S^f_c$ be a nonzero $L_K(E)$-homomorphism. By the same approach as above, $\phi(z)$ may be written in the form $0\neq \phi(z) = \sm \beta_ip_i(c)z$, where
$m\ge 1$, $p_i(c)$'s are nonzero elements in $K[c]/K[c]f(c)$, and $\beta_i$'s are paths in $E^*$ such that $\beta_ic^{\infty}$'s are distinct infinite paths in $E^{\infty}$. If $c^{\infty} \neq \beta_ic^{\infty}$ in $E^{\infty}$ for all $1 \le i \le m$, then there exists a positive integer $d$ such that $\tau_{\le d}(c^{\infty})$ and  $\tau_{\le d}(\beta_ic^{\infty})$'s are distinct paths in $E^*$, and so $(c^*)^d \beta_i c^d = 0$ for all $i$. This implies
\[\phi((c^*)^d z) = (c^*)^d\phi(z) =  (c^*)^d(\sum^m_{i=1}\beta_ip_i(c)z)= (c^*)^d(\sum^m_{i=1}\beta_ic^dp_i(c)f_1(c)^dz)= 0.\] On the other hand, 
we have that $\phi$ is an automorphism (since $\phi$ is nonzero and $S^f_c$ is simple) and $(c^*)^dz = f_1(c)^dz \neq 0$, and so 
$\phi((c^*)^dz) \neq 0$, a contradiction. This implies that there exists a $i$ such that $\beta_i c^{\infty}= c^{\infty}$  in $E^{\infty}$. Without loss of generality, we may assume that $\beta_1 c^{\infty}= c^{\infty}$. We then have that  $\beta_1 = c^t$ 
for some $t\ge 0$, and
$\tau_{\le d}(c^{\infty}) = \tau_{\le d}(\beta_1c^{\infty})$ and  $\tau_{\le d}(\beta_ic^{\infty})$ ($i = 2, \ldots, m$) are distinct paths in $E^*$, so $\tau_{\le d + l}(c^{\infty}) = \tau_{\le d +l}(\beta_1c^{\infty})$ and  $\tau_{\le d +l}(\beta_ic^{\infty})$ ($i = 2, \ldots, m$) are distinct paths in $E^*$ for all $l\ge 0$. Therefore, without loss of generality, we may assume that 
$\tau_{\le d}(c^{\infty}) = \tau_{\le d}(\beta_1c^{\infty}) = c^l$ for some $l \ge t$, and so
$\phi((c^*)^l z) = (c^*)^l\phi(z) =  (c^*)^l(c^tp_1(c)z + \sum^m_{i=2}\beta_ip_i(c)z)= (c^*)^l(c^lp_1(c)f_1(c)^{l-t}z + \sum^m_{i=2}\beta_ic^lp_i(c)f_1(c)^lz)= p_1(c)f_1(c)^{l-t}z + \sum^m_{i=2}(c^*)^l\beta_ic^lp_i(c)f_1(c)^lz = p_1(c)f_1(c)^{l-t}z$. This implies  $$\phi(z) = \phi(c^l (c^*)^lz) = c^l \phi((c^*)^l z) = c^l p_1(c)f_1(c)^{l-t}z,$$ so $\phi(z)$ can be written in the form $\phi(z) = p(c)z$, where $p(c)\in K[c]/K[c]f(c)\cong K[x]/K[x]f(x)$. Conversely, let $p(c)$ be a nonzero element in $K[c]/K[c]f(c)$. We then have $p(c)z \neq 0$ in $S^f_c$ (since $vz =z\neq 0$ and $p(c)$ is a unit in $K[c]/K[c]f(c)$) and $f(c)(p(c)z) = (f(c)p(c))z = (p(c)f(c))z= p(c)(f(c)z) = 0$, and so there exists a nonzero $L_K(E)$-homomorphism $\pi: S^f_c\longrightarrow S^f_c$ such that $\pi(z) = p(c)z$. Therefore, we have $\End_{L_K(E)}(S^f_c)\cong K[c]/K[c]f(c)\cong K[x]/K[x]f(x)$.

(2) Write $g(x) = 1 - b_1x - \cdots - b_m x^m \in K[x]$ and $c = e_1 \cdots e_t$.
Assume that $S^g_d$ is the left $L_K(E)$-module generated by $z'$ subject to $z'= (b_1d + \cdots + b_md^m)z'= d g_1(d)z'$.

($\Longrightarrow$) Assume that $\phi: S^f_c\longrightarrow S^g_d$ is a $L_K(E)$-isomorphism. Then, by the same approach as above, $\phi(z) = \sum^s_{i=1} \alpha_i a_iz'$, where $a_i$'s are nonzero element in $K[d]/K[d]g(d)$ and $\alpha_i$'s are paths in $E^*$ such that $\alpha_i d^{\infty}$'s are distinct infinite paths in $E^{\infty}$. If $c^{\infty}\neq \alpha_i d^{\infty}$ in $E^{\infty}$
for all $1\le i\le s$, then there exist a positive integer $t$ such that $(c^*)^t \alpha_i d^t = 0$ for all $1\le i\le s$. Then, since $z'=  d g_1(d)z'$, $\phi(z) = \sum^s_{i=1} \alpha_i a_iz'$ $ = \sum^s_{i=1} \alpha_i a_i d^t g_1(d)^t z'$ $ = \sum^s_{i=1} \alpha_i  d^t a_i g_1(d)^t z'$, and so $$\phi((c^*)^tz)=(c^*)^t \phi(z) = \sum^s_{i=1} (c^*)^t\alpha_i d^t a_i g_1(d)^t z'=0.$$
On the other hand, we note that $(c^*)^tz = f_1(c)^t z \neq 0$ in $S^f_c$, so $\phi((c^*)^tz)\neq 0$, since $\phi$ is a $L_K(E)$-isomorphism, a contradiction. This implies $c^{\infty}= \alpha_i d^{\infty}$ for some $1 \le i\le s$, so  $c^{\infty} \sim d^{\infty}$. Since $c$ and $d$ are simple closed paths in $E$, we must have $d = c_j := e_j\cdots e_te_1 \cdots e_{j-1}$ for some $1\le j\le t$.

Let $z'' := e_1 \cdots e_{j-1} z'$ if $j\neq 1$ and  $z'' := z'$ if $j=1$. Since $z' = (e_1 \cdots e_{j-1})^* z'' \neq 0$, $z''\neq 0$ in $S^g_d= S^g_{c_j}$. We then have $e_j \cdots e_t g(c)z'' = c_j g(c_j)z' = dg(d)z' =0$ in $S^g_d$, and so $g(c)z'' = (e_j \cdots e_t)^*e_j \cdots e_t g(c)z'' =0$ in $S^g_d$. By item (1), $S^g_d$ can be also generated by $z''$ subject to $g(c)z'' =0$.

By repeating approach described in the proof of item (1), we obtain that $\phi(z) = p(c)z''$, where $p(c)$ is a nonzero element of $K[c]/K[c]g(c)$.
We then have $0 = \phi(0) = \phi(f(c)z)= f(c)\phi(z)= f(c)(p(c)z'') = (p(c)f(c))z''$, and so $p(c)f(c)=0$ in $K[c]/K[c]g(c)$, by item (1). Since $p(c)$ is a unit in $K[c]/K[c]g(c)$, $f(c)\in K[c]g(c)$. Then, since $f$ is an irreducible polynomial in $K[x]$, we must have $f = g$. 

($\Longleftarrow$) Assume that $f = g$ and $c^{\infty} \sim d^{\infty}$. Since $c$ and $d$ are simple closed paths, $d = c_j := e_j \cdots e_te_1\cdots e_{j-1}$ for some $1\le j\le t$. Then, by repeating method described as in the direction ($\Longrightarrow$), we obtain that $S^f_c \cong S^f_d$, thus finishing the proof.
\end{proof}

We should mention the following useful remark.
 
\begin{rem} Let $K$ be a field, $E$ a graph and $c= e_1\cdots e_t$ a simple closed path in $E$ based at $v$, and let $f(x)\in \text{Irr}(K[x])$.

(1) We denote by $\Pi_c$ the set of all the following closed paths $c_1 := c,\, c_2:= e_2\cdots e_t e_1,\, \ldots, \, c_n := e_n e_1\cdots e_{n-1}$. By Theorem \ref{Rangmod}, all modules $S^f_{c_i}$ are isomorphic to each other, and for a simple closed path $d$, $S^f_{d} \cong S^f_{c}$ if and only if $d\in \Pi_c$. Consequently, if one can represent their isomorphism class by a simple module $S^f_{\Pi_c}$ isomorphic to some $S^f_{c_i}$, then $S^f_{\Pi_c}$ is well-defined and depends only on the $(f, \Pi_c)$.
	 
(2) If $f(x) = 1 - x\in K[x]$, then by Theorem~\ref{Rangmod} (1) and \cite[Theorem 2.8]{amt:eosmolpa}, we have
$S^f_c \cong L_K(E)v/L_K(E)(c -v)\cong V_{[c^{\infty}]}$ as left $L_K(E)$-modules.

(3) It was shown in the proof of Theorem \ref{Rangmod} that every element $y$ of $S^f_c$ may be written in the form $y = \sum^n_{i=1}\a_i p_i(c)z$, where $\a_i$'s are paths in $E$ such that $\alpha_ic^{\infty}$'s are distinct infinite paths in $E^{\infty}$ and $p_i(c)$'s are nonzero elements of $K[c]/K[c]f(c)$.	
\end{rem}



We next construct new classes of simple modules for the Leavitt path algebra $L_K(R_n)$ by using Theorem \ref{Rangmod}, Corollary \ref{Anicktype3} and special closed paths in $R_n$.

\begin{nota}
For any integer $n\ge 2$, we denote by $C_s(R_n)$ the set of simple closed paths of the form $c = e_{k_1}e_{k_2} \cdots e_{k_m}$, where $k_i\in \{1, 3, \ldots, n\}$ for all $1\le i\le m-1$ and $k_m =2$, in $R_n$.
\end{nota}

Let $c= e_{k_1}e_{k_2}\cdots e_{k_m}\in C_s(R_n)$, $p\in A_{R_n}(e_1, e_2)$, and $f = 1 - a_1x_1 - \cdots - a_n x^n = 1 - xf_1(x)\in \text{Irr}(K[x])$. We have a left $L_K(R_n)$-modle $S^{f,\, p}_{c}$,  which is the twisted module $(S^f_{c})^{\sigma_p}$, where $\sigma_p$  is the automorphism of $L_K(R_n)$ defined in Corollary~\ref{Anicktype3}.
Denoting by $\ast$ the module operation in $S^{f,\, p}_{c}$, we have the following useful fact.

\begin{lem}\label{Scalarmult}
Let $K$ be a field, $n\ge 2$ a positive integer, and $R_n$ the rose with $n$ petals. Let $p\in A_{R_n}(e_1, e_2)$ be an arbitrary element, $c\in C_s(R_n)$, and $f\in \rm{Irr}$$(K[x])$. Then the following statements hold:	
	
$(1)$ $c \ast y = cy + e_{k_1}\cdots e_{k_{m -1}}e_1py$ for all $y\in S^{f,\, p}_{c}$;
	
$(2)$ $(c^*)^m \ast z =  (c^*)^mz$ for all $m\ge 1$, where $z$ is a generator of the left $L_K(E)$-module $S^{f}_{c}$ which is described in Theorem \ref{Rangmod}.
\end{lem}
\begin{proof}
(1) By Corollary \ref{Anicktype3}, we have $\sigma_p(c)= c + e_{k_1}\cdots e_{k_{m -1}}e_1p$, so $c \ast y = \sigma_p(c)y = cy + e_{k_1}\cdots e_{k_{m -1}}e_1py$ for all $y\in S^{f,\, p}_{c}$, as desired.
	
(2) If $e_{k_i}\neq e_1$ for all $1\le i\le m-1$, then 	$\sigma_p(c^*)= c^*$, and so $c^* \ast z =  c^*z$, as desired. Consider the case that $e_{k_i} = e_1$ for some $1\le i \le m-1$. Let $\ell$ be the number of all elements $1\le i \le m-1$ such that $e_{k_i} = e_1$. We use induction on $\ell$ to establish the claim that $\sigma_p(c^*)c =1$ in $L_K(R_n)$. If $\ell =1$, there is a unique element $1\le i \le m-1$ such that $e_{k_i} = e_1$. We then have $\sigma_p(c^*) = e^*_{k_m}\cdots e^*_{k_{i+1}} (e^*_1-pe^*_2) e^*_{k_{i-1}}\cdots e^*_{k_1}= c^* - e^*_{k_m}\cdots e^*_{k_{i+1}} pe^*_2 e^*_{k_{i-1}}\cdots e^*_{k_1}$, and so $\sigma_p(c^*) c = (c^* - e^*_{k_m}\cdots e^*_{k_{i+1}} pe^*_2 e^*_{k_{i-1}}\cdots e^*_{k_1})c = 1$, since $e^*_2e_{k_i} = e^*_2e_1 = 0$, as desired. Now we proceed inductively. For $\ell > 1$, let $j:= \min\{i\mid 1\le i\le m-1 \text{ and } e_{k_i} = e_1\}$. We have $\sigma_p(c^*)= \sigma_p(e^*_{k_m}\cdots e^*_{k_{j+1}})(e^*_1 - pe^*_2)e^*_{k_{j-1}}\cdots e^*_{k_1}$.  It is clear that $e_{k_{j+1}}\cdots e_{k_m}\in C_s(R_n)$. Then, by the induction hypothesis, we obtain that $\sigma_p(e^*_{k_m}\cdots e^*_{k_{j+1}})e_{k_{j+1}}\cdots e_{k_m} = 1$. This implies that $\sigma_p(c^*) c= \sigma_p(e^*_{k_m}\cdots e^*_{k_{j+1}})(e^*_1 - pe^*_2)e^*_{k_{j-1}}\cdots e^*_{k_1}c = \sigma_p(e^*_{k_m}\cdots e^*_{k_{j+1}})e_{k_{j+1}}\cdots e_{k_m} = 1$, since $e^*_2e_{k_j} = e^*_2e_1 = 0$, thus showing the claim.
By induction we get that $\sigma_p((c^*)^m) c^m = \sigma_p(c^*)^m c^m = 1$ in $L_K(R_n)$ for all $m\ge 1$.

By Theorem \ref{Rangmod}, $z = cf_1(c)z$, so $z = c^mf_1(c)^mz$ and $(c^*)^mz = f_1(c)^mz$ for all $m\ge 1$. We then have $(c^*)^m \ast z= \sigma_p((c^*)^m)z = \sigma_p((c^*)^m)(c^mf_1(c)^mz) = (\sigma_p((c^*)^m) c^m)f_1(c)^mz=f_1(c)^mz = (c^*)^mz$ for all $m\ge 1$, thus finishing the proof.
\end{proof}

We are now in position to provide the first main result of this section.

\begin{thm}\label{Irrrep1}
Let $K$ be a field, $n\ge 2$ a positive integer, and $R_n$ the rose with $n$ petals. Let $p$ and $q\in A_{R_n}(e_1, e_2)$ be two arbitrary elements and $c, d\in C_s(R_n)$, and $f, g\in \rm{Irr}$$(K[x])$. Then the following holds:

$(1)$	$S^{f,\, p}_{c}$ is a simple left $L_K(R_n)$-module;
	
$(2)$ $S^{f,\, p}_{c} \cong S^{g,\, q}_{d}$ as left $L_K(R_n)$-modules if and only if $f = g$, $c=d$, and $p-q = rf(c)$ for some $r\in L_K(R_n)$;
	
$(3)$ For any simple closed path $\alpha$ in $R_n$, $S^{f,\, p}_{c} \cong S^f_{\Pi_{\alpha}}$ as left $L_K(R_n)$-modules if and only if $\alpha \in \Pi_c$ and $p = rf(c)$ for some $r\in L_K(R_n)$;
	
$(4)$ $\rm{End}$$_{L_K(R_n)}(S^{f,\, p}_{c}) \cong K[x]/K[x]f(x)$;

$(5)$ $S^{f,\, p}_{c} \cong L_K(R_n)/L_K(R_n)f(\sigma_p^{-1}(c))$.
\end{thm}
\begin{proof}
(1) It follows from that $S^f_{c}$ is a simple left $L_K(R_n)$-module (by Theorem \ref{Rangmod}) and  $\sigma_p$ is an automorphism of $L_K(R_n)$ (by Corollary~\ref{Anicktype3}).
	
(2) Assume that $\phi: S^{f,\, p}_{c} \longrightarrow S^{g,\, q}_{d}$ is a $L_K(R_n)$-isomorphism. Let $z$ and $z'$ be generators of the left $L_K(E)$-modules $S^{f,\, p}_{c}$ and $S^{g,\, q}_{d}$ which are described in Theorem \ref{Rangmod}, respectively. We then have $0\neq \phi(z) = \sum^t_{i=1}\alpha_ia_iz'$ in $S^{g,\, q}_{d}$, where $\alpha_i$'s are nonzero elements of $K[d]/K[d]g(d)$ and $\alpha_i$'s are paths in $(R_n)^*$ such that $\alpha_i d^{\infty}$'s are distinct infinite paths in $(R_n)^{\infty}$. By Theorem \ref{Rangmod} and Lemma \ref{Scalarmult}, we note that $z = c^k f_1(c)^kz$ and  $(c^*)^k\ast z = (c^*)^k z$ in $S^{f,\, p}_{c}$ for all $k\ge 1$, and $z' = d^l g_1(d)^lz'$ and  $(d^*)^l\ast z' = (d^*)^l z'$ in $S^{g,\, q}_{d}$ for all $l\ge 1$. Therefore, by repeating approach described in the proof of the direction $(\Longrightarrow)$ of Theorem~\ref{Rangmod} (2), we obtain that $c^{\infty} \sim  d^{\infty}$ and $\phi((c^*)^kz) = a z'$ for some $k\ge 1$ and a nonzero element $a\in K[d]/K[d]g(d)$. Then, since  $c^{\infty} \sim  d^{\infty}$, we have $d\in \Pi_c$, and so $c = d$, since $c$ and $d\in C_s(R_n)$. We also note that
 $\phi((c^*)^{k+1}\ast z) = \phi(c^*\ast ((c^*)^kz))= c^*\ast \phi((c^*)^kz) = c^*\ast (az') =  c^*(az')= a g_1(c)z'.$
By induction we may prove that $$\phi((c^*)^{k +i}\ast z) = \phi((c^*)^{k+i} z)= a g_1(c)^{i}z'$$ for all $i \ge 0$, where $ g_1(c)^{0}:= 1_{L_K(R_n)}$.

In $S^{f,\, p}_{c}$ we have $z= c^mf_1(c)^mz$ and $(c^*)^m= f_1(c)^mz$ for all $m\ge 1$, and
$c\ast z = c z +  e_{k_1}\cdots e_{k_{m -1}}e_1p z= c z + p'z$, where $p':= e_{k_1}\cdots e_{k_{m -1}}e_1p\in A_{R_n}(e_1, e_2)$,  so 
\begin{align*}
c\ast ((c^*)^{k+1}z) &= c\ast (f_1(c)^{k+1}z) = c (f_1(c)^{k+1}z) + p'(f_1(c)^{k+1}z)=\\&= f_1(c)^{k}z + p'(f_1(c)^{k+1}z)= (c^*)^kz + p'((c^*)^{k+1}z)=\\&=(c^*)^kz + p'\ast ((c^*)^{k+1}z),
\end{align*}
since $\sigma_p(p') = p'$ (by Corollary~\ref{Anicktype3}). This implies
\begin{align*}
\phi(c\ast ((c^*)^{k+1}z))&= \phi((c^*)^kz) + \phi(p'\ast ((c^*)^{k+1}z))= az' + p'ag_1(c)z'\\&=az' + e_{k_1}\cdots e_{k_{m -1}}e_1pag_1(c)z'.
\end{align*}
On the other hand, 
\begin{align*}
\phi(c\ast ((c^*)^{k+1}z))&= c \ast \phi((c^*)^{k+1}z)=c \ast (a g_1(c)z')= \\&= ca g_1(c)z' +  e_{k_1}\cdots e_{k_{m -1}}e_1qag_1(c)z'=\\& =  az' +  e_{k_1}\cdots e_{k_{m -1}}e_1qag_1(c)z',
\end{align*}
since $ca g_1(c)z' = a cg_1(c)z'$ and $z' = cg_1(c)z'$. From these observations, we have $$e_{k_1}\cdots e_{k_{m -1}}e_1pag_1(c)z' = e_{k_1}\cdots e_{k_{m -1}}e_1qag_1(c)z'$$ in $S^{g}_{c}$, showing that $e_{k_1}\cdots e_{k_{m -1}}e_1(p-q)ag_1(c)z' = 0$ in $S^{g}_{c}$, and hence $$(p-q)ag_1(c)z'= (e_{k_1}\cdots e_{k_{m -1}}e_1)^* e_{k_1}\cdots e_{k_{m -1}}e_1(p-q)ag_1(c)z' = 0$$ in $S^{g}_{c}$. By Theorem \ref{Rangmod} (1), we have $S^g_c \cong L_K(R_n)/L_K(R_n)g(c)$, as left $L_K(R_n)$-modules, via the map: $z\longmapsto 1+ L_K(R_n)g(c)$. Therefore, $(p-q)ag_1(c) = bg(c)$ for some $b\in L_K(R_n)$. Since $ag_1(c)$ is a unit of $K[c]/K[c]g(c)$, there exist elements $\alpha,\, \beta\in K[c]$ such that $(ag_1(c))\alpha = 1 + \beta g(c)$, and so $$bg(c)\alpha =(p-q)ag_1(c)\alpha = (p-q)(1+ \beta g(c))= p-q + (p-q)\beta g(c).$$ This implies $$p - q = (b\alpha + q\beta -p\beta)g(c)= rg(c)$$ where $r:= b\alpha + q\beta -p\beta \in L_K(R_n)$. 
	
Write $f(x) = 1 - a_1 x - \cdots - a_sx^s$. We then have $(1 - a_1 c - \cdots - a_s c^s) z = 0$ and $((c^*)^{k +s} - a_1(c^*)^{k +s-1} - \cdots - a_s(c^*)^{k})z=
(c^*)^{k +s}(1 - r_1 c - \cdots - a_sc^s) z = 0$ in $S^{f,\, p}_{c}$, and so
\begin{align*}
ag_1(c)^{s} f(c) z'&= ag_1(c)^{s}(1 - a_1 c - \cdots - a_sc^s) z' = \\&= ag_1(c)^{s} z' - a_1ag_1(c)^{s-1} z' - \cdots - a_s z' =\\&  = \phi(((c^*)^{k +s} - a_1(c^*)^{k +s-1} - \cdots - a_s(c^*)^{k})z)= \phi(0)=0
\end{align*}
in $S^{g}_{c}$. By repeating the same argument described above, we obtain that $f(c) = rg(c)$ for some $\gamma\in L_K(R_n)$. Write $\gamma = \sum^d_{i=1}k_i\alpha_i\beta^*_i$, where $k_i\in K\setminus \{0\}$ and $\alpha_i$, $\beta_i$ are paths in $R_n$. Let $m = \max\{|\alpha_i|, \, |\beta_i|\mid 1\le i\le d\}$. We then have $(c^*)^m\gamma c^m = \sum^d_{i=1}k_i(c^*)^m\alpha_i\beta^*_ic^m\in K[c]$ and $f(c)= (c^*)^mc^m f(c) = (c^*)^m f(c) c^m = (c^*)^m (\sum^d_{i=1}k_i\alpha_i\beta^*_i) g(c)c^m = (\sum^d_{i=1}k_i(c^*)^m\alpha_i\beta^*_ic^m)g(c)$ in $K[c]$, and so $f = g$, since $f, g \in \text{Irr}(K[x])$.
	
Conversely, assume that $f = g$, $c =d$ and $p- q = r f(c)$ for some $r\in L_K(R_n)$. We use induction to claim that $\sigma_q(\sigma_p^{-1})(c^m))z = c^mz$ for all $m\ge 1$. For $m= 1$, by Corollary \ref{Anicktype3}, $\sigma_p^{-1}(c) = e_{k_1}\cdots e_{k_{m-1}}e_2 - e_{k_1}\cdots e_{k_{m-1}}e_1p$, and so $\sigma_q(\sigma_p^{-1}(c)) = \sigma_q(e_{k_1}\cdots e_{k_{m-1}}e_2 ) - \sigma_q(e_{k_1}\cdots e_{k_{m-1}}e_1p)= c + e_{k_1}\cdots e_{k_{m-1}}e_1q - e_{k_1}\cdots e_{k_{m-1}}e_1p = c + e_{k_1}\cdots e_{k_{m-1}}e_1(q - p) = c - e_{k_1}\cdots e_{k_{m-1}}e_1 rf(c)$. Then,  since $f(c) z= 0$, we have $\sigma_q(\sigma_p^{-1}(c))z = cz$.
For $m > 1$, we have $\sigma_q(\sigma_p^{-1}(c^{m+1}))z = \sigma_q(\sigma_p^{-1}(c) \sigma_p^{-1}(c^{m}))z = \sigma_q(\sigma_p^{-1}(c)) \sigma_q(\sigma_p^{-1}(c^m))z = \sigma_q(\sigma_p^{-1}(c)) c^mz = (c - e_{k_1}\cdots e_{k_{m-1}}e_1 rf(c))(c^mz)= c^{m+1}z -e_{k_1}\cdots e_{k_{m-1}}e_1 rf(c)c^mz = c^{m+1}z -e_{k_1}\cdots e_{k_{m-1}}e_1 rc^mf(c)z = c^{m+1}z$, as desired. This shows that $\sigma_q(\sigma^{-1}_p(f(c)))z = f(c)z$.

We note that since $S^{f,\, p}_{c}$ is a simple left $L_K(R_n)$-module, every element of $S^{f,\, p}_{c}$ may be written in the form: $\sigma_p(s)z$, where $s\in L_K(R_n)$. Define $\phi:S^{f,\, p}_{c} \longrightarrow S^{f,\, q}_{c}$ as follows: $\sigma_p(s)z\longmapsto \sigma_q(s)z$. We claim that $\phi$ is well defined. Indeed, let $s$ and $t$ be two elements in $L_K(R_n)$ such that $\sigma_p(s)z = \sigma_p(t)z$ in $S^{f}_{c}$. By Theorem \ref{Rangmod} (1), $\sigma_p(s - t) = \sigma_p(s) -\sigma_p(t) = b f(c)$ for some $b\in L_K(R_n)$, so $s - t = \sigma^{-1}_p(b f(c))= \sigma^{-1}_p(b)\sigma^{-1}_p(f(c))$. This implies $(\sigma_q(s) -\sigma_q(t))z= \sigma_q(s-t)z= \sigma_q(\sigma^{-1}_p(b))\sigma_q(\sigma^{-1}_p(f(c)))z = \sigma_q(\sigma^{-1}_p(b))(f(c)z) = 0$ in $S^f_c$, thus proving the claim.

It is obvious that $\phi$ is a non-zero $L_K(R_n)$-homomorphism (since $\phi(z) = z$), so $\phi$ is an isomorphism.

(3) $(\Longrightarrow)$ Assume that $S^{f, \, p}_{c} \cong S^f_{\Pi_{\alpha}}$. We then have $S^{f, \, p}_{c} \cong S^f_{\alpha}$. By repeating the same method described in the proof of the direction $(\Longrightarrow)$ of Theorem \ref{Rangmod} (2), we obtain that $\alpha\in \Pi_c$. This implies 
$S^{f, \, p}_{c} \cong S^f_{c}=S^{f, \, 0}_{c}$, and so $p = r f(c)$ for some $r\in L_K(R_n)$, by item (2). 

$(\Longleftarrow)$ It immediately follows from item (2).

(4) Let $\phi: S^{f,\, p}_{c} \longrightarrow S^{f,\, p}_{c}$ is a nonzero $L_K(R_n)$-homomorphism. Since $ S^{f,\, p}_{c}$ is a simple left $L_K(R_n)$-module, $\phi$ is an isomorphism. Similar to item (2) we have $\phi((c^*)^{k} z) = a z$ for some non-zero element $a\in  K[c]/K[c]f(c)$ and some positive integer $k$.  Therefore, $\phi(r\ast ((c^*)^{k} z)) = r\ast (a z)$ for all $r\in L_K(R_n)$. Conversely, let $a$ be a nonzero element of $K[c]/K[c]f(c)$. Since  $S^{f,\, p}_{c}$ is a simple left $L_K(R_n)$-module, $S^{f,\, p}_{c} = L_K(R_n)\ast ((c^*)^{k} z)$. We claim that the map $\mu: S^{f,\, p}_{c}\longrightarrow S^{f,\, p}_{c}$, defined by $\mu(r\ast ((c^*)^{k} z) = r\ast (a z)$, is a nonzero $L_K(R_n)$-homomorphism. Indeed, assume that $r\ast (c^*)^{k} z = s\ast (c^*)^{k} z$, where $r, s\in L_K(R_n)$. We then have $\sigma_p(r) f_1(c)^k z = \sigma_p(s) f_1(c)^k z$ in $S^{f}_{c}$. By Theorem \ref{Rangmod} (1), we obtain that $(\sigma_p(r) - \sigma_p(s))f_1(c)^k = b f(c)$ for some $b\in L_K(R_n)$, and so $r\ast (a z) - s\ast (a z) = \sigma_p(r) a  z - \sigma_p(s)a  z = (\sigma_p(r) -\sigma_p(s)) a  z = (\sigma_p(r) -\sigma_p(s)) a c^k f_1(c)^k z= (\sigma_p(r) -\sigma_p(s))f_1(c)^k a c^k z = b f(c) a c^k z = b a c^k f(c)  z =  0$ (since $f(c) z =0$), that means, $\mu(r\ast ((c^*)^{k} z) = \mu(s\ast ((c^*)^{k} z)$. This implies that $\mu$ is well defined. It is not hard to check that $\mu$ is a $L_K(R_n)$-homomorphism. From these observations, we have  $\rm{End}$$_{L_K(R_n)}(S^{f,\, p}_{c}) \cong K[c]/K[c] f(c)\cong  K[x]/K[x] f(x)$.

(5) We first note that $S^{f,\, p}_{c} = L_K(R_n)\ast z$, {\it i.e.}, every element of $S^{f,\, p}_{c}$ is of the form $r\ast z= \sigma_p(r)z$, where $r\in L_K(R_n)$. We next computer $\text{ann}_{L_K(R_n)}(z)$. Indeed, let $r\in \text{ann}_{L_K(R_n)}(z)$. We then have $\sigma_p(r)z = r\ast z =0$ in $S^f_c$. By Theorem \ref{Rangmod} (1), $\sigma_p(r) = b f(c)$ for some $b \in L_K(R_n)$, and so $r = \sigma_p^{-1}(b) \sigma_p^{-1}(f(c)) = \sigma_p^{-1}(b) f(\sigma_p^{-1}(c))$, since $\sigma_p^{-1}$ is an endomorphism of the $K$-algebra $L_K(R_n)$. This implies $\text{ann}_{L_K(R_n)}(z) \subseteq L_K(R_n) f(\sigma_p^{-1}(c))$. 

Conversely, assume that $r \in  L_K(R_n) f(\sigma_p^{-1}(c))$, {\it i.e.}, $r = \beta f(\sigma_p^{-1}(c))$, where $\beta \in L_K(R_n)$. We then have $r\ast z = \sigma_p(r) z = \sigma_p(\beta) f(c) z = 0$ (since $f(c) z=0$), and so $r\in \text{ann}_{L_K(R_n)}(z)$, showing that $L_K(R_n) f(\sigma_p^{-1}(c))\subseteq\text{ann}_{L_K(R_n)}(z)$. Hence $\text{ann}_{L_K(R_n)}(z) = L_K(R_n) f(\sigma_p^{-1}(c))$. This implies $$S^{f,\, p}_{c} \cong L_K(R_n)/L_K(R_n)f(\sigma_p^{-1}(c)),$$ thus finishing the proof.	
\end{proof}

For clarification, we illustrate Theorem \ref{Irrrep1} by presenting the following example.

\begin{exas}
Let $K$ be a field and $R_2$ the rose with $2$ petals.  We then have
$C_s(R_2) = \{e_2,\, e^m_1 e_2\mid m\in \mathbb{Z},\, m\ge 1\}$, and $A_{R_2}(e_1, e_2)$ is the $K$-subalgebra of $L_K(R_2)$ generated by $v, \, e_1, \, e^*_2$, that means, $$A_{R_2}(e_1, e_2) = \{\sum^n_{i =1}k_i e^{m_i}_1 (e^*_2)^{l_i}\mid n\ge 1,\, k_i \in K,\, m_i, l_i\ge 0 \},$$ where $e^0_1 = v = (e^*_2)^0$, and $K[e_1] \subset A_{R_2}(e_1, e_2)$. Let $f(x) = 1 - x\in \text{Irr}(K[x])$. By the grading on $L_K(R_2)$, we always have $p \neq r (1- c)$ for all $c\in C_s(R_2)$, $p\in K[e_1]\setminus \{0\}$ and $r\in L_K(R_2)$, and so the set $$\{S^{f,\, p}_{c} \mid c \in C_s(R_2), \, p \in K[e_1]\}$$ consists of pairwise non-isomorphic simple left $L_K(R_2)$-modules, by Theorem \ref{Irrrep1}. 

Let $p = e_1$ and $q= e_1 e^*_2\in  A_{R_2}(e_1, e_2)$. We then have $q - p = e_1 e^*_2 - e_1 = e_1 e^*_2 (1 - e_2)$, so $S^{f,\, p}_{e_2} \cong S^{f,\, q}_{e_2}$ as left $L_K(R_2)$-modules, by Theorem \ref{Irrrep1}. 
\end{exas}

Using Theorems \ref{Chenmod}, \ref{Rangmod} and \ref{Irrrep1} we obtain a list of pairwise non-isomorphic simple modules for the Leavitt path algebra $L_K(R_n)$. Before doing so, we need some useful notions. For each pair $(f, c) \in \text{Irr}(K[x])\times C_s(R_n)$, we define a relation $\equiv_{f, c}$ on $A_{R_n}(e_1, e_n)$ as follows.
For all $p, q\in A_{R_n}(e_1, e_n)$, $p\equiv_{f, c} q$ if and only if $p-q = rf(c)$ for some $r\in L_K(R_n)$. It is obvious that $\equiv_{f, c}$ is an equivalence on $A_{R_n}(e_1, e_n)$. We denote by $[p]$ the $\equiv_{f, c}$ equivalent class of $p$.

\begin{thm}\label{Irrrep2} 
Let $K$ be a field, $n\ge 2$ a positive integer, and $R_n$ the rose with $n$ petals. Then, the following set 
$$\{V_{[\alpha]}\mid \alpha\in (R_n)^{\infty}_{irr}\} \sqcup \{S^f_{\Pi_c}\mid c\in SCP(R_n),\, f \in \textnormal{Irr}(K[x])\}\, \sqcup$$
\[\sqcup\, \{S^{f,\, p}_{d} \mid d\in C_s(R_n),\, f \in \textnormal{Irr}(K[x]),\, [0]\neq [p]\in A_{R_n}(e_1, e_2)/\equiv_{f, d}\}\]
consists of pairwise non-isomorphic simple left $L_K(R_n)$-modules.	
\end{thm}
\begin{proof}
All $V_{[\alpha]}$ ($\alpha \in (R_n)^{\infty}_{irr}$) are pairwise non-isomorphic by Theorem \ref{Chenmod}. All $S^f_{\Pi_c}$ ($ c\in SCP(R_n),\, f \in \textnormal{Irr}(K[x])$) are pairwise non-isomorphic by Theorem \ref{Rangmod}. All $S^{f,\, p}_{d}$ ($d\in C_s(R_n),\, f \in \textnormal{Irr}(K[x]),\, [0]\neq [p]\in A_{R_n}(e_1, e_2)/\equiv_{f, d}$) are pairwise non-isomorphic by Theorem \ref{Irrrep1}. 
	
By Theorem \ref{Chenmod} and \ref{Irrrep1} respectively, $S^f_{\Pi_c}$ ($ c\in SCP(R_n),\, f \in \textnormal{Irr}(K[x])$) and $S^{f,\, p}_{d}$ ($d\in C_s(R_n),\, f \in \textnormal{Irr}(K[x]),\, [0]\neq p\in A_{R_n}(e_1, e_2)$) are finitely presented. While by \cite[Corollary 3.5]{anhnam}, $V_{[\alpha]}$ is not finitely presented for all $\alpha\in (R_n)^{\infty}_{irr}$. Therefore each $V_{[\alpha]}$ is neither isomorphic to any $S^f_{\Pi_c}$ nor any $S^{f,\, p}_{d}$.

By Theorem \ref{Irrrep1} (3), each $S^{f,\, p}_{d}$ ($d\in C_s(R_n),\, f \in \textnormal{Irr}(K[x]),\, [0]\neq [p]\in A_{R_n}(e_1, e_2)/\equiv_{f, d}$) is not isomorphic to any  $S^f_{\Pi_c}$ ($ c\in SCP(R_n),\, f \in \textnormal{Irr}(K[x])$), thus finishing the proof.
\end{proof}

\vskip 0.5 cm \vskip 0.5cm {
	
\end{document}